\documentclass{amsart}

\usepackage{graphicx}
\usepackage{amsthm}
\usepackage{amssymb}
\usepackage{amsmath}
\usepackage{amsfonts}
\usepackage{amssymb}
\usepackage{microtype}

\newtheorem{theorem}{Theorem}[section]
\newtheorem{lem}[theorem]{Lemma}
\newtheorem{cor}[theorem]{Corollary}
\newtheorem{main}{Theorem}
\newtheorem{maincor}[main]{Corollary}

\newcommand{\inverse}[1]{#1^{-1}}

\newcommand{\centerof}[1]{\mathbf{Z}(#1)}
\newcommand{\centralizer}[2]{\mathbf{C}_{#1}(#2)}

\newcommand{\id}{\{\,1\,\}}

\begin{document}

\title{The Cyclic Graph of a Z-group}

\author[D.G.\ Costanzo]{David G.\ Costanzo}
\address{Department of Mathematical Sciences, Kent State University, Kent, OH 44242, USA}
\email{dcostan2@kent.edu}

\author[M.L.\ Lewis]{Mark L.\ Lewis}
\address{Department of Mathematical Sciences, Kent State University, Kent, OH 44242, USA}
\email{lewis@math.kent.edu}

\author[S.\ Schmidt]{Stefano Schmidt}
\address{Department of Mathematics, Columbia University, New York, NY 10027, USA}
\email{sas2393@columbia.edu}

\author[E.\ Tsegaye]{Eyob Tsegaye}
\address{Department of Mathematics, Stanford University, Stanford, California 94305, USA}
\email{eytsegay@stanford.edu}

\author[G. Udell]{Gabe Udell}
\address{Department of Mathematics, Pomona College, Claremont, CA 91711, USA}
\email{grua2017@mymail.pomona.edu}


\maketitle

\begin{abstract}
For a group $G$, we define a graph $\Delta(G)$ by letting $G^{\#}=G\setminus\id$ be the set of vertices and by drawing an edge between distinct elements $x,y\in G^{\#}$ if and only if the subgroup $\langle x,y\rangle$ is cyclic.
Recall that a $Z$-group is a group where every Sylow subgroup is cyclic.
In this short note, we investigate $\Delta(G)$ for a $Z$-group $G$.
\end{abstract}
	
\section{Introduction}

The groups under consideration in this note are finite.
Let $G$ be a group, and define a graph $\Delta(G)$ associated with $G$ as follows.
Take $G^{\#}=G\setminus\id$ as the vertex set.
Then, draw an edge between distinct vertices $x,y\in G^{\#}$ if and only if the subgroup $\langle x,y\rangle$ is cyclic.
We shall refer to $\Delta(G)$ as the \textit{cyclic graph} of $G$, although we note that the graph $\Delta(G)$ has also been called the \textit{deleted enhanced power graph}.
See, for example, \cite{bera}.
The \textit{enhanced power graph} includes the identity element as a vertex, and so the enhanced power graph of a group is always connected.
A brief investigation of this graph was undertaken in \cite{cameron}.

The cyclic graph of a group $G$ was investigated in \cite{imp} and \cite{implewis}.
In those papers, classification results were obtained under the assumption that the connected components of $\Delta(G)$ were complete graphs.
In our previous paper \cite{directprod}, we studied the cyclic graph of a direct product.

Next, we mention another graph that can be attached to a group.
Let $G$ be a nonabelian group.
The \textit{commuting graph} of $G$, denoted by $\Gamma(G)$, is the graph whose vertices are the non-central elements of $G$ and whose edges connected distinct vertices $x$ and $y$ if and only if $xy=yx$.
The commuting graph of a finite solvable group with trivial center was classified in \cite{parker}.

Recall that a group is called a \textit{$Z$-group} if every Sylow subgroup is cyclic.
Observe that a Frobenius complement of odd order is a $Z$-group, and so is any group of square-free order.
Our focus in this short note is the graph $\Delta(G)$ for a $Z$-group $G$.
We have been able to characterize the disconnectedness of $\Delta(G)$.

\begin{main}
Let $G$ be a $Z$-group.
Then $\Delta(G)$ is disconnected if and only if $G$ is a Frobenius group.
\end{main}

Next, if we assume that the graph $\Delta(G)$ is connected for a $Z$-group $G$, then a diameter bound follows.

\begin{main}
If $G$ is a $Z$-group and $\Delta(G)$ is connected, then $\textnormal{diam}(\Delta(G))\le 4$.
\end{main}

The next result describes a relationship between the graph $\Delta(G)$ and the subgroup $\centerof{G}$ for a $Z$-group $G$.

\begin{main}
If $G$ is a $Z$-group, then $\textnormal{diam}(\Delta(G))\le 2$ if and only if $\centerof{G}\neq \id$.
\end{main}

Following \cite{bera}, a vertex $z$ in $\Delta(G)$ is called a \textit{dominating vertex} if $z$ is adjacent to every vertex in $\Delta(G)\setminus\{z\}$.
The terms \textit{complete vertex}, \textit{cone vertex}, and \textit{universal vertex} have also been used as synonyms for a dominating vertex.
If the graph $\Delta(G)$ has a dominating vertex, we shall say that $\Delta(G)$ is \textit{dominatable}.
In the proof of the previous theorem, we end up establishing the existence of a dominating vertex.
We point out a necessary and sufficient condition for a dominating vertex in $\Delta(G)$ to exist, which answers a request in \cite{bera} for a characterization of a group with a dominatable cyclic graph.

\begin{main}
Let $G$ be a group, $g\in G$, and $\pi=\pi(o(g))$.
Write $g=\prod_{p\in\pi}g_{p}$, where each $g_{p}$ is a $p$-element for $p\in\pi$ and $g_{p}g_{q}=g_{q}g_{p}$ for all $p,q\in\pi$.
Then $g$ is a dominating vertex for $\Delta(G)$ if and only if, for each $p\in\pi$, a Sylow $p$-subgroup $P$ of $G$ is cyclic or generalized quaternion and $\langle g_{p}\rangle\le P\cap\centerof{G}$.
\end{main}

As a corollary, we offer a generalization of Theorem 3.2 in \cite{bera}.

\begin{maincor}
For a nilpotent group $G$, the graph $\Delta(G)$ is dominatable if and only if $G$ has a cyclic or generalized quaternion Sylow subgroup.
\end{maincor}

Let $G$ be a $Z$-group, and let $x,y\in G^{\#}$ be distinct.
If $x$ is adjacent to $y$ in $\Delta(G)$, then $xy=yx$.
In fact, the converse is true too.
So, in particular, if $\centerof{G}=\id$, then $\Gamma(G)$ and $\Delta(G)$ are the same graph.
In light of the previous results, we obtain the following corollary concerning the commuting graph of a $Z$-group with trivial center.

\begin{maincor}
If $G$ is a $Z$-group with $\centerof{G}=\id$ and $G$ is not a Frobenius group, then $\Gamma(G)$, the commuting graph of $G$, is connected with diameter $3$ or $4$.
\end{maincor}

This research was conducted during an REU at Kent State University with the funding of NSF Grant DMS-1653002.
We thank the NSF for their support.
The first, third, fourth, and fifth authors also thank the faculty and staff at Kent State University for their hospitality.

\section{Notation and Preliminaries}

Let $G$ be a group, and let $x,y\in G$.
We shall write $x\approx y$ to indicate that the subgroup $\langle x,y\rangle$ is cyclic.
If $n$ is a positive integer, then $\pi(n)$ denotes the set of prime divisors of $n$.
For a group $G$, set $\pi(G)=\pi(|G|)$.
Fix a set of prime numbers $\pi$.
An element $x\in G$ is called a \textit{$\pi$-element} if every prime divisor of $o(x)$ is a member of $\pi$.
If every prime divisor of $o(x)$ lies outside of $\pi$, then $x$ is called a \textit{$\pi'$-element}.
In this case where $\pi=\{p\}$, we use the terms $p$-element and $p'$-element.
The set of prime numbers shall be denoted by $\mathbb{P}$.

Let $G$ be a group.
Notice that if $x,y\in G^{\#}$ are commuting elements with coprime orders, then $x\approx y$.
This fact gives us a useful way to build paths in $\Delta(G)$.
We also mention that conjugation preserve adjacency in $\Delta(G)$: specifically, if $x,y\in G^{\#}$ with $x\approx y$, then $x^{g}\approx y^{g}$ for each $g\in G$.

A graph related to the cyclic graph is the \textit{commuting graph}, which is defined as follows.
Let $G$ be a nonabelian group.
The commuting graph $\Gamma(G)$ is the graph whose vertices are the noncentral elements of $G$ and whose edges connect distinct nonidentity elements $x$ and $y$ if and only if $xy=yx$.
Taking the noncentral elements of $G$ as the vertices for $\Gamma(G)$ is fairly standard, although variations on the vertex set do exist.
If $G$ is a $Z$-group with a trivial center, then the vertex set of $\Gamma(G)$ is the same as the vertex set of $\Delta(G)$.
In fact, the edge sets are the same too; the following lemma also appears as a part of Theorem 30 in \cite{cameron}.

\begin{lem}\label{gamma=delta}
If $G$ is a $Z$-group with $\centerof{G}=\id$, then $\Delta(G)=\Gamma(G)$.
\end{lem}

\begin{proof}
If $x,y\in G$ with $x\approx y$, then clearly $xy=yx$.
But notice that if $xy=yx$, then $\langle x,y\rangle$ is an abelian $Z$-group, which is therefore cyclic.
Hence $x\approx y$.
\end{proof}

Next, we make a few remarks about $Z$-groups.
Many properties of $Z$-groups are known.
For example, if $G$ is a $Z$-group, then $G$ is $p$-nilpotent for the smallest prime divisor $p$ of $|G|$.
We also know that $Z$-groups are solvable.
The specific results that we need in this paper are encapsulated in the following theorem.

\begin{theorem}[\cite{rose}, Theorem 10.26]
If $G$ is a $Z$-group, then the derived subgroup $G'$ is cyclic and the factor group $G/G'$ is cyclic.
Moreover, $G'$ is a Hall subgroup of $G$.
\end{theorem}

Finally, we need to make an observation about Frobenius groups.
Recall that a group $G$ is a \textit{Frobenius group} if $G$ has a nontrivial proper subgroup $H$ such that $H\cap H^{g}=\id$ for each $g\in G\setminus H$.
The subgroup $H$ is called a \textit{Frobenius complement}.
Now, let $G$ be a Frobenius group with Frobenius complement $H$.
Frobenius groups are centerless, and so $\Gamma(G)$ and $\Delta(G)$ have the same vertex set.
In particular, $\Delta(G)$ is a spanning subgraph of $\Gamma(G)$.
Because $\centralizer{G}{h}\le H$ for each $h\in H^{\#}$, the graph $\Gamma(G)$ is disconnected.
(This fact appears as Lemma 3.1 in \cite{parker}.)
Hence $\Delta(G)$ is disconnected as well.

\section{Main Results}\label{sec: main results}

Our first theorem provides a necessary and sufficient condition for the cyclic graph of $Z$-group $G$ to be disconnected.
Additionally, a diameter bound of $\Delta(G)$ is available under the assumption that $\Delta(G)$ is connected.

\begin{theorem}\label{disconnected}
Let $G$ be a $Z$-group.
Then $\Delta(G)$ is disconnected if and only if $G$ is a Frobenius group.
Moreover, if $\Delta(G)$ is connected, then $\textnormal{diam}(\Delta(G))\le 4$.
\end{theorem}

\begin{proof}
Frobenius groups have disconnected cyclic graphs.
To prove the converse, assume that $G$ is not a Frobenius group.
We shall establish the connectedness of $\Delta(G)$.

Abelian $Z$-groups are cyclic, and so we may assume that $G$ is nonabelian.
Hence $\id < G'<G$.
If $\centralizer{G}{g}\le G'$ for each $g\in (G')^{\#}$, then $G$ is a Frobenius group with kernel $G'$, contrary to our hypothesis.
Hence there exists some $g_{0}\in (G')^{\#}$ with $\centralizer{G}{g_{0}}\nleq G'$.
Let $H$ be a complement for $G'$ in $G$.
Fix $x\in\centralizer{G}{g_{0}}\setminus G'$, and write $x=yh$ for $y\in G'$ and $h\in H$.
Then $g_{0}^{yh}=g_{0}^{x}=g_{0}$, and so $g_{0}^{\inverse{h}}=g_{0}^{y}=g_{0}$.
It follows that $h\in\centralizer{H}{g_{0}}$.

Now, let $g\in G^{\#}$.
If $\pi(o(g))\cap\pi(G')\neq\emptyset$, then let $p\in\pi(o(g))\cap\pi(G')$.
For a suitable integer $n$, $o(g^{n})=p$; hence $g^{n}\in G'$, and $g\approx g^{n}\approx g_{0}$.
Otherwise, $\pi(o(g))\cap\pi(G')=\emptyset$ and $g\in H^{a}$ for some $a\in G'$.
Note that $h^{a}\approx g_{0}^{a}=g_{0}$ as $h\approx g_{0}$ and conjugation preserves adjacency.
Hence $g\approx h^{a}\approx g_{0}$.
The result follows.
\end{proof}

The group \texttt{SmallGroup(60,7)} furnishes an example of a $Z$-group with connected cyclic graph of diameter $4$, and so the bound in the previous theorem is sharp.
We mention a few more examples.
The group \texttt{SmallGroup(210,2)} is a $Z$-group with connected cyclic graph of diameter $3$, and, finally, \texttt{SmallGroup(60,3)} provides an example of a $Z$-group with connected cyclic graph of diameter $2$.

\begin{figure}[h]
\centering

\begin{minipage}{0.8\textwidth}
\centering
\includegraphics[width=0.5\linewidth]{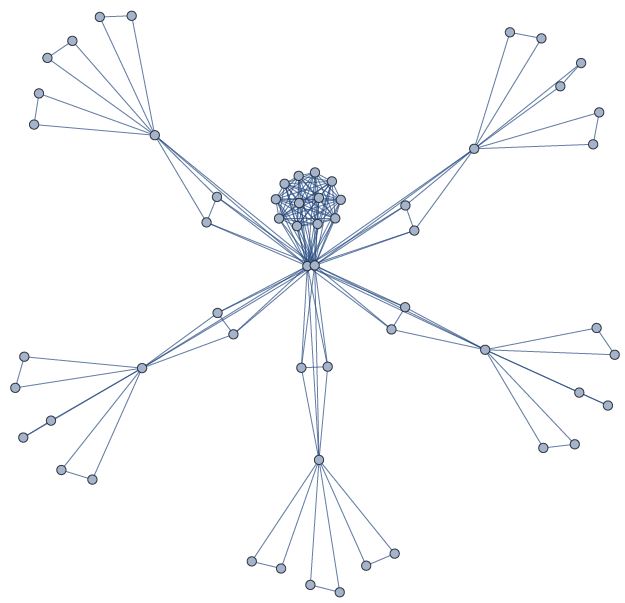}
\caption{\texttt{Cyclic Graph of SmallGroup(60,7)}}
\label{fig:60_7}
\end{minipage}\hfill

\begin{minipage}{0.8\textwidth}
\centering
\includegraphics[width=0.5\linewidth]{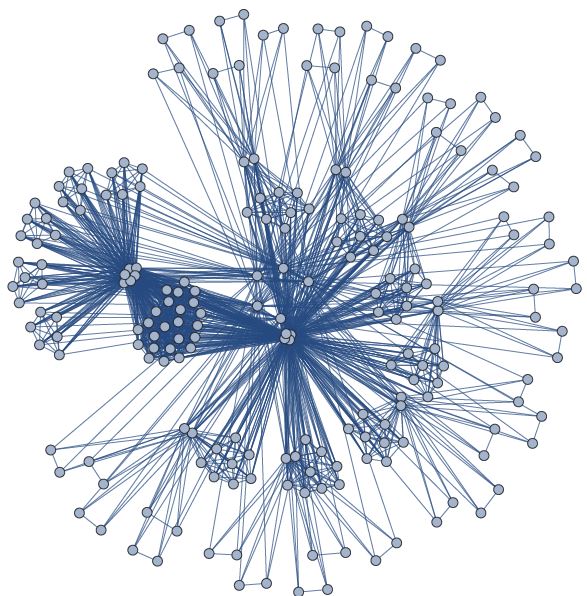}
\caption{\texttt{Cyclic Graph of SmallGroup(210,2)}}
\label{fig:2102}
\end{minipage}\hfill

\begin{minipage}{0.8\textwidth}
\centering
\includegraphics[width=0.5\linewidth]{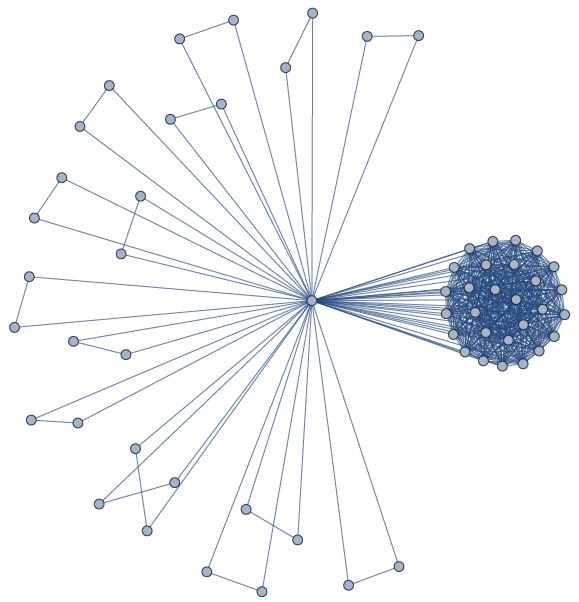}
\caption{\texttt{Cyclic Graph of SmallGroup(60,3)}}
\label{fig:60_3}
\end{minipage}

\end{figure}

The next theorem highlights a connection between the subgroup $\centerof{G}$ and the graph $\Delta(G)$ for a $Z$-group $G$.

\begin{theorem}\label{nontrivialcenterdiameter2}
If $G$ is a $Z$-group, then $\textnormal{diam}(\Delta(G))\le 2$ if and only if $\centerof{G}\neq\id$.
\end{theorem}

\begin{proof}
Assume that $\centerof{G}\neq\id$.
Fix $z\in\centerof{G}^{\#}$ with $o(z)=p$, a prime.
Since $\langle z\rangle$ is a normal $p$-subgroup of $G$ and every Sylow $p$-subgroup of $G$ is cyclic, $\langle z \rangle$ is the unique subgroup of $G$ with order $p$.
If $g\in G^{\#}$ and $p$ divides $o(g)$, then $\langle z\rangle\le\langle g\rangle$.
Hence $g\approx z$.
Otherwise, $o(g)$ is a $p'$-number, and so $g$ and $z$ are commuting elements with coprime orders.
Again, $g\approx z$.

Assume that $\textnormal{diam}(\Delta(G))\le 2$.
Let $H$ be a complement of $G'$.
Set $G'=\langle x \rangle$.
As $G/G'\cong H$, the subgroup $H$ is cyclic.
Set $H=\langle h\rangle$.
If $x\approx h$, then $G$ is abelian, and so $\centerof{G}=G\neq\id$.
Otherwise, $x\approx z\approx h$ for some $z\in G^{\#}$.
Now, $G'=\langle x\rangle\le\centralizer{G}{z}$ and $H=\langle h\rangle\le\centralizer{G}{z}$.
It follows that $G=G'H\le\centralizer{G}{z}$.
Hence $z\in\centerof{G}^{\#}$.
\end{proof}

Let $G$ be a group, and recall that a vertex $z$ in $\Delta(G)$ is called a \textit{dominating vertex} if $z\approx g$ for each $g\in\Delta(G)\setminus\{z\}$.
A dominating vertex appears in the previous proof, and the following theorem highlights a necessary and sufficient condition for such a vertex to exist.

\begin{theorem}\label{dominating}
The cyclic graph of a group $G$ has a dominating vertex if and only if $G$ has a unique subgroup of order $p$ for some prime $p$ and this subgroup is central.
\end{theorem}

\begin{proof}
Let $c$ be a dominating vertex of $\Delta(G)$.
For suitable integer $t$, $o(c^{t})=p\in\mathbb{P}$.
For each $g\in\Delta(G)\setminus\{c^{t}\}$, $$\langle c^{t}, g\rangle\le\langle c,g\rangle,$$ and so $c^{t}$ is a dominating vertex as well.
Note that $\langle c^{t}\rangle$ is a central subgroup of prime order.
Suppose that $\langle y\rangle$ has order $p$.
The subgroup $\langle c^{t},y\rangle$ is cyclic, and therefore has a unique subgroup of order $p$.
Hence $\langle c^{t}\rangle=\langle y\rangle$.

Conversely, suppose that $\langle z\rangle$ is a central subgroup of order $p\in\mathbb{P}$ and, further, that $\langle z\rangle$ is the \textit{unique} subgroup of order $p$.
If $g\in G$ is $p'$-element, then $z\approx g$ since $z$ and $g$ are commuting elements with coprime orders.
If $p$ divides $o(g)$, then $|\langle g^{t}\rangle|=p$ for suitable integer $t$.
Our uniqueness hypothesis forces $\langle z\rangle=\langle g^{t}\rangle$.
Again, $z\approx g$.
The element $z$ is a dominating vertex.
\end{proof}

The relationship between the existence of a dominating vertex for the cyclic graph of a group and the Sylow subgroup structure of the group can be developed a bit further.
Let $G$ be a group, $g\in G$, and $\pi=\pi(o(g))$.
Using Theorem 5.1.5 in \cite{wrscott}, write $g=\prod_{p\in\pi}g_{p}$, where each $g_{p}$ is a $p$-element for $p\in\pi$ and $g_{p}g_{q}=g_{q}g_{p}$ for all $p,q\in\pi$.
Then $g$ is a dominating vertex for $\Delta(G)$ if and only if, for each $p\in\pi$, a Sylow $p$-subgroup $P$ of $G$ is cyclic or generalized quaternion and $\langle g_{p}\rangle\le P\cap\centerof{G}$.
We remark that this result strengthens Theorem \ref{dominating}---and has essentially the same proof.

Bera and Bhuniya \cite{bera} showed that if $G$ is abelian, then  $\Delta(G)$ is dominatable if and only if $G$ has a cyclic Sylow subgroup.
We generalize this result.

\begin{cor}
If $G$ is a nilpotent group, then $\Delta(G)$ is dominatable if and only if $G$ has a cyclic or generalized quaternion Sylow subgroup.
\end{cor}

\begin{proof}
If $\Delta(G)$ has a dominating vertex, then, by Theorem \ref{dominating}, $G$ has a unique subgroup $\langle x\rangle$ of prime order, say $p$, that is contained in $\centerof{G}$.
It is easy to check that if $P$ is the Sylow $p$-subgroup of $G$, then $\langle x\rangle$ is the unique subgroup of $P$ order $p$; hence, $P$ is cyclic or generalized quaternion.

Conversely, suppose that $G$ has a Sylow $p$-subgroup $P$ that is cyclic or generalized quaternion.
Let $\langle z\rangle$ be the unique subgroup of $G$ of order $p$.
Let $g\in G^{\#}$.
If $o(g)$ is a $p'$-number, then $z\approx g$ as $z$ and $g$ are therefore commuting elements with coprime orders.
If $p$ divides $o(g)$, then $|\langle g^{s}\rangle|=p$ for suitable integer $s$.
Hence $\langle z\rangle =\langle g^{s}\rangle$, and so $z$ is a power of $g$.
Again, $z\approx g$.
We conclude that $z$ is a dominating vertex.
\end{proof}

As mentioned previously, if $G$ is a $Z$-group with $\centerof{G}=\id$, then $\Gamma(G)=\Delta(G)$. 
We now obtain information about the commuting graph $\Gamma(G)$ of a $Z$-group $G$ with trivial center.

\begin{cor}
Let $G$ be a $Z$-group with $\centerof{G}=\id$.
If $G$ is not a Frobenius group, then $\Gamma(G)$ is connected with $\textnormal{diam}(\Gamma(G))\in\{3,4\}$.
\end{cor}

\begin{proof}
By Lemma \ref{gamma=delta}, $\Gamma(G)=\Delta(G)$.
Since $G$ is not a Frobenius group, Theorem \ref{disconnected} yields that $\Gamma(G)$ is connected.
Theorem \ref{nontrivialcenterdiameter2} gives us that $\textnormal{diam}(\Gamma(G))\ge 3$.
Finally, an application of Theorem \ref{disconnected} implies that $\textnormal{diam}(\Gamma(G))$ is either $3$ or $4$.
\end{proof}

\end{document}